\documentclass[11pt]{article}
\usepackage[margin=1in]{geometry}
\usepackage{dsfont,epstopdf}
\usepackage{graphicx}
\usepackage{amssymb,amsrefs,amsmath,verbatim,amsthm}
\usepackage[english]{babel}
\usepackage[linesnumbered,ruled]{algorithm2e}
\usepackage{enumerate}
\usepackage{hyperref}
\usepackage[affil-it]{authblk}
\newtheorem{theorem}{Theorem}
\newtheorem{proposition}[theorem]{Proposition}
\newtheorem{corollary}[theorem]{Corollary}
\newtheorem{definition}[theorem]{Definition}
\newtheorem{lemma}[theorem]{Lemma}

\begin{document}
\title{Blaschke Decompositions on Weighted Hardy Spaces}
\date{}
\author{Stephen D. Farnham \thanks{Syracuse University Department of Mathematics, Syracuse, NY 13244, USA; email: sdfarnha@syr.edu}}

\maketitle

\abstract{Recently, several papers have considered a nonlinear analogue of Fourier series in signal analysis, referred to as either nonlinear phase unwinding or adaptive Fourier decomposition. In these processes, a signal is represented as the real component of a complex function $F: \partial\mathbb{D}\to \mathbb{C}$, and by performing an iterative method to obtain a sequence of Blaschke decompositions, the signal can be approximated using only a few terms. To better understand the convergence of these methods, the study of Blaschke decompositions on weighted Hardy spaces was studied by Coifman and Steinerberger, under the assumption that the complex valued function $F$ has an analytic extension to $\mathbb{D}_{1+\epsilon}$ for some $\epsilon>0$. This provided bounds on weighted Hardy norms involving a function and its Blaschke decomposition. That work also noted that in many examples, the nonlinear unwinding series of a function converges at an exponential rate to the original signal, which when coupled with an efficient algorithm to perform a Blaschke decomposition, has lead to a new and efficient way to approximate signals. 

In this work, we continue the study of Blaschke decompositions on weighted Hardy Spaces for functions in the larger space $\mathcal{H}^2(\mathbb{D})$ under the assumption that the function has finitely many roots in $\mathbb{D}$. By studying the growth rate of the weights, we improve the bounds provided by Coifman and Steinerberger. This provides us with new insights into Blaschke decompositions on classical function spaces including the Hardy-Sobolev spaces and weighted Bergman spaces. Further, we state a sufficient condition on the weights for our improved bounds to hold for any function in the Hardy space, $\mathcal{H}^2(\mathbb{D})$. These results may help to better explain why the exponential convergence of the unwinding series is seen in many numerical examples.
\medskip


\medskip

}

\section{Introduction}
In many fields, the Hardy Space $\mathcal{H}^p(\mathbb{D})$, where $1\leq p \leq \infty$, has been studied due to its well behaved nature when compared to the larger Lebesgue space, $L^p(\mathbb{D})$. One of the most well known results for $\mathcal{H}^p(\mathbb{D})$ spaces is the decomposition theorem. Simply put, given a function $F\in \mathcal{H}^p(\mathbb{D})$, we can decompose
\begin{equation}
F = B\cdot G
\end{equation}
where $|B(z)| \leq 1$ and $G(z) \neq 0$ for any $z\in \mathbb{D}$. In this factorization, the function $B$ is a Blaschke product with the same zeros as $F$ in $\mathbb{D}$, and $G$ is a function that is also in the space $\mathcal{H}^p(\mathbb{D})$. In the past decade, this theorem was utilized to create an iterative method to express a function, $F\in\mathcal{H}^2(\mathbb{D})$ as a summation involving Blaschke products \cite{Nahon:2010}. These summations are commonly referred to as the ``unwinding series" of $F$.

%
%
This idea inspired research in both signal processing and complex analysis as seen in  \cite{C.S.:2017, C.S.:2018, Qian:2010,Qian:2017} and many other works. In 2017, it was pointed out in \cite{C.S.:2017} that weighted Hardy spaces are a well suited environment to study the effect of Blaschke decompositions, and in particular can be used to show the convergence of the unwinding series. Given arbitrary, monotone increasing weights, several results were provided for functions $F$ that are analytic in $\mathbb{D}_{1+\epsilon}$ for some $\epsilon>0$. 
 
 Using the same notion of weighted Hardy spaces, this paper seeks to enhance the main results of \cite{C.S.:2017} in two ways. Firstly, by controlling the growth rate of the weights, we create tighter bounds on the inequalities seen in \cite{C.S.:2017}, for functions in $\mathcal{H}^2(\mathbb{D})$ with a finite number of roots in $\mathbb{D}$. From a practical standpoint, this tells us how Blaschke decompositions affect a majority of functions in many well known spaces, including Hardy-Sobolev spaces and weighted Bergman spaces. Second, we show that for certain bounded weights, we can extend our inequalities to functions that are simply in $\mathcal{H}^2(\mathbb{D})$, regardless of the cardinality of the functions' roots in $\mathbb{D}$.

 \subsection{Acknowledgement}
The results of this work are part of the Ph.D. dissertation of the author.  I would like to thank my advisors, Dr. Loredana Lanzani and Dr. Lixin Shen, for all of their support through the writing of this paper and my thesis. 

\subsection{Main Results}
 We begin by recalling the definition of weighted $\mathcal{H}^2$ spaces, denoted $X_\gamma$ and $Y_\gamma$ found in \cite{C.S.:2017}.

\begin{definition}
Let $\{\gamma_n\}\not \equiv 0$ be a monotone increasing sequence of real numbers that satisfies $\gamma_0=0$. Given a function $F \in \mathcal{H}^2$, we say that $F$ belongs to the space $X_\gamma$ if

\begin{equation}\label{Xgammanorm}
\|F(z)\|_{X_\gamma}^2 = \|\sum_{j\geq 0} a_j z^j \|_{X_\gamma}^2 := \sum_{j\geq 0} \gamma_j |a_j|^2 < \infty.
\end{equation}

Moreover, we define the space $Y_\gamma$ to be the set of functions that satisfy
\begin{equation}\label{Ygammanorm}
\|F(z)\|_{Y_\gamma}^2 = \|\sum_{j\geq 0} a_j z^j \|_{Y_\gamma}^2 := \sum_{j\geq 0} (\gamma_{j+1}-\gamma_j) |a_j|^2 < \infty.
\end{equation}
\end{definition}

These weighted spaces allow us to prove results on many classical spaces in a generalized way. For example, if we choose a sequence $\gamma_n$, bounded above by some constant $C$, such that $\gamma_1 = c,$ where $ c>0$, then the space $X_\gamma$ will be equivalent to the space of $\mathcal{H}^2$ functions satisfying $F(0)=0$, as
\[ \frac{1}{C} \|F\|_{X_\gamma}^2 \leq \|F\|_{\mathcal{H}^2}^2 \leq \frac{1}{c}\|F\|_{X_\gamma}^2. \]

Similarly, if we choose the sequence $\gamma_n=n$, then the space $X_\gamma$ will be equivalent to the smaller Dirichlet space, which is the space of functions in $\mathcal{H}^2$ with derivatives that are also in $\mathcal{H}^2$.

As was pointed out in \cite{C.S.:2017}, the $X_\gamma$ and $Y_\gamma$ spaces allow for a better understanding of the Blaschke decomposition compared to the unweighted $\mathcal{H}^2$ space. To see this, if we are given a function $F$ with a Blaschke decomposition, $F=B\cdot G$, then we have the equality
\[\|F\|_{\mathcal{H}^2}^2 =\|G\|_{\mathcal{H}^2}^2.\]

 Coifman and Steinerberger demonstrated that any function, $F$, that is analytic in $\mathbb{D}_{1+\epsilon}$ and has a root at  $\alpha \in \mathbb{D}$, we have the inequality
\begin{equation}\label{eq:CSinequality}
\|G\|_{X_\gamma}^2 \leq \|F\|_{X_\gamma}^2 - (1-|\alpha|^2) \left \|\frac{G(e^i\cdot)}{1-\overline\alpha e^{i\cdot}}\right \|_{Y_\gamma}^2.
\end{equation}

Moreover, from the proof of this inequality, the authors of \cite{C.S.:2017} were able to provide an enhanced version of \eqref{eq:CSinequality} for the special case $\gamma_n=n$. The result is stated as Corollary 2 in \cite{C.S.:2017}, and tells us that for functions $F\in \mathcal{D}$, the Dirichlet space, with $m$ roots\footnote{The statement of the Corollary does not seem to preclude infinitely many roots as long as $F\in \mathcal{H}^\infty.$}, labeled $\alpha_1,\dots \alpha_m$, we have the identity 
\begin{equation} \label{eq:Cor2}
 \|G\|_{\mathcal{D}}^2=\|F\|_{\mathcal{D}}^2 - \sum_{j=1}^m(1-|\alpha_j|^2)  \left \|\frac{G(\cdot)}{\cdot-\alpha_j}\right\|_{\mathcal{H}^2}^2.
 \end{equation}

 In this paper, we improve upon the two aforementioned results by studying two different types of sequences, $\gamma_n$: ones that increase at an increasing rate, and ones that increase at a decreasing rate. These two types of sequences are of theoretical and practical importance, and by specifying the rate of increase, we are able to create tighter bounds on the  $X_\gamma$ norm of $G$ to obtain inequalities similar to \eqref{eq:Cor2} for more generalized spaces. Further, by imposing additional conditions on the growth rate of the sequence $\gamma_n$, we are able to remove the assumption that the function $F$ has a finite number of roots in $\mathbb{D}$, and obtain a converging infinite series on the right hand side of \eqref{eq:Cor2}.

The first result of this paper, Theorem \ref{thm:MR1} below, investigates the case where the sequence $\{\gamma_n\}$ is increasing at an increasing rate. In other words, by defining
\[ \Gamma_n:= \gamma_{n+1}-\gamma_n,\]
for any $n\geq 0$, we require $\Gamma_n$ to be a monotone increasing sequence. Equivalently stated, we require
\begin{equation}\label{eq:increase}
\forall n\geq 0, \quad \gamma_{n+2}-2\gamma_{n+1}+\gamma_n \geq 0.
\end{equation}
 This condition provides us with several classical function spaces. As was previously mentioned, if $\gamma_n=n$, then the space $X_\gamma$ is equivalent to the Dirichlet space, denoted $ \mathcal{D}$,  and $Y_\gamma$ is $\mathcal{H}^2$. Moreover, if $\gamma_n=n^2$, then $X_\gamma$ is equivalent to the Hardy-Sobolev space, denoted $W^{1,2}$, and $Y_\gamma$ is equivalent to $\mathcal{D}$. 

We point out that functions in these types of weighted Hardy spaces may have a finite or an infinite number of roots in $\mathbb{D}$. In the latter case, a sufficient condition was given by Shapiro and Shields in \cite{Shapiro:1962} for an infinite set of points, $a_n\in \mathbb{D}$, to be the zero set of a function in weighted Hardy spaces. Essentially, given $X_\gamma$, the growth rate of $\gamma_n$ dictates the minimal convergence rate of points to $\partial \mathbb{D}$ for those points to be a zero set of a function in the space. Due to this dependence, we limit our first theorem to functions with a finite number of roots in $\mathbb{D}$ and leave the convergence in the case when $F$ has an infinite number of roots as an open question. This gives us the following result.

\begin{theorem}\label{thm:MR1} Suppose that $\gamma_n$ is a monotone increasing satisfying $\gamma_0=0$ and \eqref{eq:increase}. For functions $F \in X_\gamma$ with a finite number of zeros inside the unit disc labeled $\alpha_1,\alpha_2, \dots, \alpha_m$, we have

\begin{equation} \label{eq:thm1eq1}
\|G(e^{i \cdot})\|_{X_\gamma}^2  \leq \|F(e^{i \cdot})\|_{X_\gamma}^2 - \sum_{j=1}^m \left( (1-|\alpha_j|^2) \left \| \frac{G(e^{i \cdot})}{1-\overline{\alpha_j}e^{i \cdot}}\right \|_{Y_\gamma}^2 \right).
\end{equation}
 \end{theorem}
 
 From this Theorem, we arrive at the following two Corollaries which encompass the well known Dirichlet and Hardy-Sobolev spaces, defined in the next section.
 
To begin, as an immediate application of Theorem \ref{thm:MR1}, we obtain an alternate proof of \eqref{eq:Cor2}.
 \begin{corollary}\label{thm:MR1Cor1}
Suppose that $\gamma_n$ is monotone increasing such that for any $n\geq 0$, $\gamma_{n+1}-\gamma_n\equiv C,$ for some constant $C>0$. For functions $F \in X_\gamma$ with a finite number of zeros inside the unit disc labeled $\{ \alpha_1,\alpha_2, \dots, \alpha_m \}$, we have the identity
 \begin{equation}  \label{eq:thm1eq2}
  \|G(e^{i \cdot})\|_{X_\gamma}^2 = \|F(e^{i \cdot})\|_{X_\gamma}^2-\sum_{j=1}^m \left( (1-|\alpha_j|^2) \left \| \frac{G(e^{i \cdot})}{1-\overline{\alpha_j}e^{i \cdot}}\right \|_{Y_\gamma}^2 \right).
\end{equation}
\end{corollary}
 
Our second Corollary provides a new inequality on the Hardy-Sobolev norm of the function $G$ involving both the Dirichlet and Hardy norms of $G$. In this Corollary, we provide the bound for the space $W^{1,2}$, but note that the same techniques can be used to create bounds on the spaces $W^{s,2}$, where $s\in \mathbb{N}$. 
  \begin{corollary}\label{thm:MR1Cor2}
   Let $F\in W^{1,2}$ have Blaschke decomposition $F=B\cdot G$. Suppose $F$ has a finite number of roots in $\mathbb{D}$ labeled $\alpha_1,\alpha_2,\dots, \alpha_m$. Then
   \[ \|G\|_{W^{1,2}}^2 \leq \|F\|_{W^{1,2}}^2 - \sum_{j=1}^m (1-|\alpha_j|^2)\left[2 \left \| \frac{G(e^{i\cdot})}{1-\overline{\alpha_j}e^{i\cdot}}\right\|^2_{\mathcal{D}} -\left \| \frac{G(e^{i\cdot})}{1-\overline{\alpha_j}e^{i\cdot}}\right\|^2_{\mathcal{H}^2} \right]. \]
   \end{corollary}
\medskip

 From here, we investigate the case when $\{\gamma_n\}$ is increasing at a decreasing rate. That is, for any $n\geq 0$, we require $\Gamma_n$ to be monotone decreasing. In other words,

\begin{equation}\label{eq:decrease}
 \forall n\geq 0, \quad \gamma_{n+2}-2\gamma_{n+1}+\gamma_n \leq 0.
\end{equation}
 As an example, we may consider
\[\gamma_n=\sum_{k=1}^n \frac{1}{k^\beta},\]
for some $ \beta \in \mathbb{N}$. If $\beta=1$, then the space $X_\gamma$ will be equivalent to the space of functions that have Fourier coefficients that satisfy
\[\sum \log(n)|a_n|^2 <\infty,\]
and if $\beta > 1$, then $X_\gamma$ will be equivalent to $\mathcal{H}^2$. In either case, $Y_\gamma$ will be a $\beta$-weighted Bergman space, a topic of interest in \cite{Weighted:2018}. 

We begin our study of spaces $X_\gamma$, where $\gamma_n$ satisfies \eqref{eq:decrease} with a result regarding functions with a finite number of roots in $\mathbb{D}$.

 \begin{theorem}\label{thm:MR2} Suppose that $\gamma_n$ is a monotone increasing sequence satisfying \eqref{eq:decrease}. For functions $F \in X_\gamma$ with a finite number of zeros inside $\mathbb{D}$ labeled $\alpha_1, \alpha_2,\dots,\alpha_m$, we have
\begin{equation} \label{eq:thm2eq1}
\|G(e^{i \cdot})\|_{X_\gamma}^2  \leq \|F(e^{i \cdot})\|_{X_\gamma}^2 - \sum_{j=1}^m \left( (1-|\alpha_j|^2) \left \| \frac{F(e^{i \cdot})}{e^{i\cdot}-\alpha_j}\right \|_{Y_\gamma}^2 \right).
\end{equation}

\end{theorem}

Similar to our first Theorem, this result connects the $X_\gamma$ norm of $F$ and $G$ by using all of the roots of $F$ in $\mathbb{D}$. However, we no longer have an expression involving $G$ on the right hand side of the inequality. Our next goal is to extend this result to functions with an infinite number of roots in $\mathbb{D}$. To prove such a result, we add two additional conditions to the sequence $\gamma_n$: boundedness and the convergence rate of the sequence to its limit. This gives us the following result.

 \begin{theorem}\label{thm:MR3} Suppose that $\gamma_n \nearrow M$ is a bounded monotone increasing sequence satisfying \eqref{eq:decrease} and $\sum_{n\geq 0}M-\gamma_n < \infty$. For any function $F \in \mathcal{H}^2$ with zeros inside the unit disc labeled in increasing order of magnitude, $\alpha_j$ for $j\in J$, we have 
 \[ \sum_{j\in J}  (1-|\alpha_j|^2) \left \| \frac{F(e^{i \cdot})}{e^{i\cdot}-\alpha_j}\right \|_{Y_\gamma}^2<\infty\]
 and 
\begin{equation} \label{eq:thm3eq1}
\|G(e^{i \cdot})\|_{X_\gamma}^2  \leq \|F(e^{i \cdot})\|_{X_\gamma}^2 - \sum_{j\in J} \left( (1-|\alpha_j|^2) \left \| \frac{F(e^{i \cdot})}{e^{i\cdot}-\alpha_j}\right \|_{Y_\gamma}^2 \right).
\end{equation}

\end{theorem}
\vspace{.5 in}

Within the proofs of Theorem \ref{thm:MR1}, Theorem \ref{thm:MR2}, and Theorem \ref{thm:MR3}, we also obtain an identity (seen later as \eqref{eq:mainidentity}) that can be used to show how a Blaschke decomposition redistributes the magnitude of the Fourier coefficients of a function $F\in \mathcal{H}^2$  . Simply put, if $F=B\cdot G$ has $m$ roots in $\mathbb{D}$ labeled $\alpha_1,\dots \alpha_m$, where
\[F(z) = \sum_{n=0}^\infty a_n z^n \quad \text{and} \quad G(z) = \sum_{n=0}^\infty b_n z^n,\]
then for any $k> 0$,  a result by Qian in \cite{Qian:2014} states
\begin{equation}\label{eq:Qianineq}
\sum_{n=k}^\infty |b_n|^2 \leq \sum_{n=k}^\infty |a_n|^2.
\end{equation}
In our work, by selecting the sequence
\begin{equation}
\gamma_n = \begin{cases}
0 & n<k\\
1 & n\geq k
\end{cases}
 \end{equation}
 and applying \eqref{eq:mainidentity} with the appropriate substitutions, we obtain the following identity, an improvement of the inequality \eqref{eq:Qianineq}.
\begin{equation}
\sum_{n=k}^\infty |b_n|^2 = \left(\sum_{n=k}^\infty |a_n|^2\right) - \frac{1}{k!}\sum_{j=1}^m \left((1-|\alpha_j|^2) \left| \frac{d}{d^k}\left[ \frac{F(\cdot)}{\prod_{\ell=1}^j (1-\overline{\alpha_\ell}\cdot)}\right](0)\right|^2 \right).
\end{equation}

The remainder of this paper is organized as follows:

 In section 2, technical definitions and a review of the literature are provided.
 
 In section 3, we prove the main results.

\section{Background and Definitions}
In this section, we briefly summarize existing works on the topic, and provide definitions and notations that will be used throughout the paper.

\subsection{Definitions and Notation}

We begin this section with the definition of the Hardy space on the complex unit disc, $\mathcal{H}^p(\mathbb{D})$, which will also be denoted $\mathcal{H}^p$ for simplicity.

The space $\mathcal{H}^p$ for $0 < p < \infty$ is the collection of analytic functions $F$ on the open unit disk that satisfy
\begin{equation}\label{eq:Hpnorm}
\|F\|_{\mathcal{H}^p}:=\sup_{0 < r < 1} \Big( \frac{1}{2 \pi} \int_0^{2\pi} | F(r e^{i \theta})|^p d \theta \Big) ^{\frac{1}{p}} <  \infty .
\end{equation}

Of significant importance is the Hardy space $\mathcal{H}^2$. In this space, one can determine the $\mathcal{H}^2$ norm of a function using only its Fourier coefficients. That is, given $F\in \mathcal{H}^2$, if for all $z\in \mathbb{D}$ we express 
\[F(z) = \sum_{n= 0}^\infty a_n z^n,\]
 then 
\begin{equation}\label{eq:H2norm}
\|F\|^2_{\mathcal{H}^2} = \sum_{n=0}^\infty |a_n|^2. 
\end{equation}
The proof of this result can be found in \cite{Zygmund:2002}.

Next, we define some previously mentioned Hilbert spaces, which are examples of weighted Hardy spaces. 

The \textit{Dirichlet Space}, denoted $\mathcal{D}$, is the space of $\mathcal{H}^2$ functions who's derivatives are also in $\mathcal{H}^2$. Functions in $\mathcal{D}$ also satisfy the identity
\begin{equation}\label{eq:Dirnorm}
\|F\|_\mathcal{D}^2 := \sum_{n=0}^\infty (n+1)|a_n|^2.
\end{equation}

The \textit{Hardy-Sobolev Spaces}, denoted $W^{s,2}$, are the spaces of $\mathcal{H}^2$ functions with $s$ weak derivatives in $\mathcal{H}^2$. These functions have associated norms that satisfy the identity
\begin{equation}
\|F\|_{W^{s,2}}^2 = \sum_{n=0}^\infty (n^{2}+1)^s|a_n|^2.
\end{equation}

As mentioned in the previous section, the Decomposition Theorem states that a function $F \in \mathcal{H}^p$ can be decomposed into the product of two functions
\begin{equation}\label{eq:decomposition}
F=B\cdot G.
\end{equation}
The second function, $G$, will be in the space $\mathcal{H}^p$, have no zeros in $\mathbb{D}$, and will satisfy $\|G\|_{\mathcal{H}^p}=\|F\|_{\mathcal{H}^p}$ . We may further impose the requirement that $G(0)$ is real valued and nonnegative, but for this paper, we make no such requirement as it will not affect any of our normed spaces.

The first function, $B$, will be unimodular, that is $|B| = 1$ a.e. on $\partial \mathbb{D}$, and will have the same zeros as $F$ inside $\mathbb{D}$. Moreover, it will be a Blaschke product.

A function $B$ is a Blaschke product if it is of the form  
\[B(z) = e^{i \phi} z^m \prod_{j \in J} \frac{\alpha_j-z}{1-\bar \alpha_j z}, \]
where $\phi \in [0,2\pi),$ $m$ is finite, $J = \{1,2,3,\dots, N\}$, where $N$ can be finite or infinity, and for any $j \in J, 0<|\alpha_j|<1 $. If $N=\infty$, we require an addition constraint on the zero set called the Blaschke condition. That is, if $\{\alpha_j\}$ is an infinite set, then 
\begin{equation} \label{eq:BlaschkeCondition}
\sum_{j=0}^\infty (1-|\alpha_j|) < \infty.
\end{equation}
Simply put, the Blaschke condition states that the set of zeros of an $\mathcal{H}^p$ function must accumulate to $\partial \mathbb{D}$ in a controlled way.

With these definitions and results, we now briefly review the existing literature.

\subsection{Nonlinear Phase Unwinding}
In the PhD thesis of Nahon \cite{Nahon:2010}, the concept of performing an unwinding series was first introduced. The idea is that given a $2 \pi$ periodic, real valued signal, $s$, we can use the Hilbert transform to create a function $F\in\mathcal{H}^2(\mathbb{D})$ whose real part agrees with the signal on $\partial \mathbb{D}$. That is, 
\[F=s+i H(s),\]
where $H(s)$ is the Hilbert transform of $s$.  From here, the unwinding series can be produced in the following way.

We begin with the decomposition
\[F=B_0 \cdot G_0.\]
Since $G_0\in \mathcal{H}^2$, by adding and subtracting the term $G_0(0)$, we can introduce a root at the origin for the function $G_0-G_0(0)$. This implies that we can decompose 
\[G_0-G_0(0)=B_1\cdot G_1.\]
Similarly, for any $n\geq 0$, we can iteratively define 
\[G_n-G_n(0) = B_{n+1}\cdot G_{n+1}.\]
 With all of this, through adding and subtracting the terms $G_j(0)$, for $0\leq j \leq n$, we can expand $F$ into its partial unwinding series
 \begin{equation}
  F = G_0(0)B_0 + G_1(0) B_0B_1 + \dots + G_n(0) \prod_{j=1}^n B_j + \prod_{j=1}^n B_j( G_n-G_n(0)).
  \end{equation}

With this series, numerical experiments were performed and early results showed that the partial unwinding series of a function converged at an exponential rate to $F$. This provided a method of approximating the original signal, $s$, using the real part of the partial unwinding series.

Several years later, Coifman and Steinerberger further investigated the unwinding series of functions in the two articles \cite{C.S.:2017,C.S.:2018}.

In the first paper, the definitions of the weighted Hardy spaces, $X_\gamma$ and $Y_\gamma$, were introduced and the results mentioned in the previous section were proven. This was the main inspiration of this work. 

In the second paper, the unwinding series was studied for functions with small (in the $L^2(\partial\mathbb{D})$ sense) antiholomorphic components and an elementary result about the equivalence of the Fourier series and unwinding series for a certain class of functions was shown.

Since then, Coifman and Peyriere have shown the convergence of the unwinding series for any $\mathcal{H}^p$ function with $p\geq 1$ in \cite{Coif:2018}
. While the convergence has been shown, the arguments are based on invariant spaces and do not tell us about the rate of convergence of the unwinding series.

%

\subsection{Adaptive Fourier Decomposition}
While nonlinear phase unwinding was being studied, Qian et al. have also worked on the related topic of Adaptive Fourier Decomposition (AFD). The main distinction between these two procedures is that at each step of AFD, instead of adding and subtracting the term $G_j(0)$ (as was done in the unwinding series), an algorithm seeks the ''optimal" point $a_j\in \mathbb{D}$ that will provide the best finite approximation if we add and subtract $G_j(a_j))$. The existence of the optimal points $a_j$, for each step $j$, has been shown in \cite{Qian:2010}, however there are no closed formulas for the explicit computation of such points (See \cite{Qian:2017} for related results.). Therefore, there is a computational cost in approximating the optimal points that is avoided in the unwinding series. We are not aware of results that compare the convergence rate of the AFD algorithm with the convergence rate of the unwinding series.

\section{Proof of Main Results}
In this section, we develop the theory necessary to prove Theorems \ref{thm:MR1}, \ref{thm:MR2}, and \ref{thm:MR3}. This section will be broken into three parts for clarity. 

In the first part, see Section \ref{Part1}, we expand upon the relationship between the functions $F$ and $G$ in the decomposition theorem to see that reflecting the roots of $F$ across $\partial \mathbb{D}$ provides a method of producing $G$. With this knowledge, we study how the act of reflecting a root in $\mathbb{D}$ across the unit circle affects the $X_\gamma$ norm of a function, seen in Proposition \ref{thm:reflectroot}. To end this part, we state Corollary \ref{thm:singlereflection} which provides an identity for the case when $F$ has a single root in $\mathbb{D}$. 

In the second part, see Section \ref{Part2}, we investigate the case when $F$ has finitely many roots in $\mathbb{D}$. We begin by defining intermediate functions (seen in Definition \ref{def:Fj}) that can be viewed as partial decompositions. Using these functions, we invoke Proposition \ref{thm:reflectroot} to obtain an identity seen in Lemma \ref{thm:nreflections}, which connects the $X_\gamma$ norms of $F$ and $G$. From there, we state and prove Lemma \ref{thm:Ygammabound}, which provides bounds on the $Y_\gamma$ norm of functions based on the growth rate of the sequence $\gamma_n$. With this, we have the tools necessary to handle all functions with a finite number of zeros in $\mathbb{D}$, and prove Theorem \ref{thm:MR1}, Theorem \ref{thm:MR2}, Corollary \ref{thm:MR1Cor1}, and Corollary \ref{thm:MR1Cor2}.

 In the final part, see Section \ref{Part3}, we prove the convergence of the inequality in Theorem \ref{thm:MR2} for functions with infinitely many roots in $\mathbb{D}$ at the cost of imposing further conditions on $\gamma_n$ (see Lemma \ref{thm:FmtoG} and Lemma \ref{thm:boundedGsup}). From there, we directly prove Theorem~\ref{thm:MR3}.


\subsection{Part 1: Reflecting a Root Across $\partial \mathbb{D}$}\label{Part1}

To begin, suppose that a function $F$ has a finite number of roots in $\mathbb{D}$ labeled in increasing order of magnitudes $\alpha_1, \alpha_2,\dots,\alpha_m$, where the roots need not be distinct. Then for $z\in \mathbb{D}$, we can express

\[F(z) = \left(\prod_{j=1}^m \frac{z-\alpha_j}{1-\overline{\alpha_j}z}\right) \cdot G(z).\]

One way of generating the function $G$ from $F$ is by replacing each term in the Blaschke product $z-\alpha_j$ with the term $1-\overline{\alpha_j}z$. This can be viewed as the reflection of each zero of $F$ across the complex unit circle. When all zeros are reflected, we will have $m$ removable singularities, all outside $\mathbb{D}$, so we will have a function equivalent to $G$ on $\mathbb{D}$. 

In the case where $F$ has infinitely many roots in $\mathbb{D}$, we can apply the same technique, and rely on several Hardy Space results to show that this process will provide a sequence of functions that converge to $G$ uniformly on compact subsets of $\mathbb{D}$. This will be further discussed in the third part of this section. 

With this knowledge, we want to investigate how the act of reflecting a single root across the unit circle changes the $X_\gamma$ norm of a function. To understand how a single reflection works, we first define the following operator.

\begin{definition}\label{def:Halphaphi}
Suppose that $F\in \mathcal{H}^2$ has a root at $\alpha \in \mathbb{D}$, that is, $F(\alpha)=0$ and for all $z\in \mathbb{D}$, we can express $F(z)=(z-\alpha)H_\alpha(z)$. We define $\phi_\alpha$ be the operator that acts on functions in $\mathcal{H}^2$ with roots at $\alpha$ and satisfies 
\begin{equation}\label{eq:phialpha}
\phi_\alpha(F(\cdot)) = \phi_\alpha((\cdot-\alpha)H_{\alpha}(\cdot)):= 
(1-\overline{\alpha}\cdot)H_{\alpha}(\cdot).
\end{equation}

\end{definition}

This definition tells us that the operator $\phi_\alpha$ only affects a single root of a function. In the case where a function has a root at $\alpha$ of higher multiplicity, this operator will reduce the multiplicity of the root by one. With this definition, we have the following result on the $\mathcal{H}^2$ norm of $H_\alpha$. 

\begin{lemma} \label{thm:Hbound}
If $F\in \mathcal{H}^2$ satisfies $F(\alpha)=0$ where $|\alpha|<1$, then $H_\alpha\in \mathcal{H}^2$. Further, we have the inequality 
\[\|H_\alpha\|_{\mathcal{H}^2} \le \frac{2}{1-|\alpha|} \|F\|_{\mathcal{H}^2}.\]
\end{lemma}

\begin{proof}
By definition, we know that 
\begin{equation}\label{eq:H2norm}
\|F\|_{\mathcal{H}^2}=\limsup_{0\le r <1} \left( \frac{1}{2\pi} \int_{0}^{2\pi} |F(re^{i\theta})|^2 d\theta \right)^{\frac{1}{2}}=M<\infty.
\end{equation}
Since $\alpha$ is a root of $F$, we know that $\frac{F(\cdot)}{( \cdot - \alpha)}=H_\alpha$ has a removable singularity at $\alpha$, so because $H_\alpha$ can be treated as an analytic function in $\mathbb{D}$,
\[\limsup_{0\le r <1} \int_0^{2\pi} \left|\frac{F(re^{i\theta})}{re^{i\theta}-\alpha}\right|^2 d\theta= \lim_{r\to 1^-}\int_0^{2\pi} \left|\frac{F(re^{i\theta})}{re^{i\theta}-\alpha}\right|^2 d\theta. \]
Since $|\alpha|<1$, we know that if $r\ge \frac{1+|\alpha|}{2}$, for any $\theta \in [0,2\pi]$,
\[\left| re^{i\theta}-\alpha \right| \geq \frac{1-|\alpha|}{2}. \]
Therefore, 
\begin{eqnarray*}
2\pi\|H_\alpha\|_{\mathcal{H}^2}^2 &= &\lim_{r\to 1^-} \int_0^{2\pi} \left|\frac{F(re^{i\theta})}{re^{i\theta}-\alpha}\right|^2 d\theta\\
& \leq & \left(\frac{2}{1-|\alpha|}\right)^2 \lim_{r\to 1^-} \int_{0}^{2\pi} |F(re^{i\theta})|^2 d\theta\\
 &\leq & \left(\frac{2}{1-|\alpha|}\right)^2 2\pi \|F\|_{\mathcal{H}^2}^2 < \infty.
\end{eqnarray*}
Dividing each side by $2\pi$ and taking square roots gives us the result.
\end{proof}

With this lemma proved, we know that the function $H_\alpha$ can be represented with a power series whose coefficients are in $\ell^2$. This allows us to properly study how the operator, $\phi_\alpha$, affects the $X_\gamma$ norm of functions and leads us to the following proposition, which is a generalization of a result in \cite{C.S.:2017}.

\begin{proposition} \label{thm:reflectroot}
 Let $F\in X_\gamma$ satisfy $F(\alpha)=0$ for some $\alpha \in \mathbb{D}$, so that for all $z\in \mathbb{D}$, $F(z)=(z-\alpha) H_\alpha(z)$. With $\phi_\alpha$ defined in \eqref{eq:phialpha}, we have the following results: 
\begin{enumerate}
\item{ $\phi_\alpha(F) \in X_\gamma$ and $H_\alpha \in Y_\gamma$ }

\item{ $  \|\phi_\alpha(F)\|_{X_\gamma}^2=  \|F\|_{X_\gamma}^2- (1-|\alpha|^2)\|H_\alpha\|_{Y_\gamma}^2. $}
\end{enumerate}

\end{proposition}

\begin{proof}
Given $F\in X_\gamma$, we know that $X_\gamma \subseteq \mathcal{H}^2$, so it follows from Lemma~\ref{thm:Hbound} that $H_\alpha \in \mathcal{H}^2$. Therefore, we know that for all $z\in \mathbb{D}$ we may represent 
\begin{equation*}
H_\alpha(z) = \sum_{j=0}^\infty a_j z^j, \quad \text{where } \sum_{j=0}^\infty |a_j|^2<\infty.
\end{equation*}
With this notation, we know that 
\[F(z)=(z-\alpha)H_\alpha(z) \quad \text{and }\quad  \phi_\alpha(F(z)) = (1-\overline\alpha z)H_\alpha(z),\]
so we can express
\[F(z) = (z-\alpha)  \sum_{j=0}^\infty a_j z^j = -\alpha a_0 + \sum_{j=1}^\infty (a_{j-1}-\alpha a_j)z^j,\]
and 
\[\phi_\alpha(F(z)) = (1-\overline\alpha z) \sum_{j=0}^\infty a_j z^j = a_0 + \sum_{j=1}^\infty (a_{j}-\overline\alpha a_{j-1})z^j.\]

With these expressions, we can compute the $X_\gamma$ norm of each function, and obtain
\begin{align}
 \|F\|_{X_\gamma}^2 = \gamma_0|\alpha a_0|^2 + \sum_{j=1}^\infty \gamma_j|a_{j-1}-\alpha a_j|^2,\\
 \|\phi_\alpha(F)\|_{X_\gamma}^2 = \gamma_0|a_0|^2 + \sum_{j=1}^\infty \gamma_j |a_{j}-\overline\alpha a_{j-1}|^2.
\end{align}
From here, if we can show that $\|F\|_{X_\gamma}^2 -  \|\phi_\alpha(F)\|_{X_\gamma}^2\geq 0$, then this will imply that $\phi_\alpha(F) \in X_\gamma$. Towards this end, we consider the following finite difference:
\[ \gamma_0|\alpha a_0|^2 + \sum_{j=1}^N \gamma_j|a_{j-1}-\alpha a_j|^2 -  \left(\gamma_0|a_0|^2 - \sum_{j=1}^N \gamma_j |a_{j}+\overline\alpha a_{j-1}|^2\right). \]
For each $j\in \{1,2,\dots,N\}$, we know that 
\[|a_{j-1}-\alpha a_j|^2 -  |a_{j}-\overline\alpha a_{j-1}|^2 =|a_{j-1}|^2 + |\alpha|^2|a_j|^2 -|a_j|^2 - |\alpha|^2||a_{j-1}|^2=(1-|\alpha|^2)\left(|a_{j-1}|^2-|a_j|^2\right), \]
because $\mathcal{R}e(a_{j-1}\alpha a_j) = \mathcal{R}e(a_j \overline\alpha a_{j-1})$. Therefore,
\begin{eqnarray*}
\gamma_0|\alpha a_0|^2 + \sum_{j=1}^N \gamma_j|a_{j-1}-\alpha a_j|^2 -  \gamma_0|a_0|^2 - \sum_{j=1}^N \gamma_j |a_{j}-\overline\alpha a_{j-1}|^2\\
= -\gamma_0(1-|\alpha|^2)|a_0|^2 + (1-|\alpha|^2) \sum_{j=1}^N \gamma_j \left(|a_{j-1}|^2-|a_j|^2\right) \\
=(1-|\alpha|^2) \left(\left[\sum_{j=0}^{N-1} (\gamma_{j+1}-\gamma_j)|a_j|^2\right] + \gamma_N|a_N|^2 \right). \\
\end{eqnarray*}
Thus, by passing limits into the summation, we have
\[\|F\|_{X_\gamma}^2 -  \|\phi_\alpha(F)\|_{X_\gamma}^2 = \lim_{N\to \infty} (1-|\alpha|^2) \left( \left[\sum_{j=0}^{N-1}(\gamma_{j+1}-\gamma_j)|a_j|^2\right] + \gamma_N|a_N|^2 \right). \]
 
Since $\gamma_n$ is an increasing sequence and $\|F\|_{X_\gamma}^2 < \infty$ we know that for $a_n\neq 0$,
\[0=\lim_{n\to \infty} \gamma_n |a_{n-1}-\alpha a_n|^2 = \lim_{n\to \infty} \gamma_n |a_n|^2 \left| \frac{a_{n-1}}{a_n}-\alpha\right|^2. \]
Since  $H_\alpha \in \mathcal{H}^2$, we know $|a_n| \to 0$, which implies that  $ \lim_{n\to \infty} \left| \frac{a_{n-1}}{a_n}-\alpha\right|^2 \neq 0$, so
 \[\lim_{N\to \infty} \gamma_N|a_N|^2 = 0.\]
By definition,
\[\lim_{N\to \infty} \sum_{j=0}^{N-1}(\gamma_{j+1}-\gamma_j)|a_j|^2 = \|H_\alpha\|_{Y_\gamma}^2,\]
so we have the identity 
\[\|F\|_{X_\gamma}^2 -  \|\phi_\alpha(F)\|_{X_\gamma}^2 = (1-|\alpha|^2)\|H_\alpha\|_{Y_\gamma}^2. \]

Since  $(1-|\alpha|^2)\|H_\alpha\|_{Y_\gamma}^2 \ge 0,$ this immediately tells us that $\|\phi_\alpha(F)\|_{X_\gamma}^2 \leq  \|F\|_{X_\gamma}^2 < \infty$, and that $\|H_\alpha\|_{Y_\gamma}^2 < \infty$, which proves (1). 

Lastly, by rearranging the terms, we have

\[ \|\phi_\alpha(F)\|_{X_\gamma}^2  =  \|F\|_{X_\gamma}^2- (1-|\alpha|^2)\|H_\alpha\|_{Y_\gamma}^2,\]
which completes the proof.

\end{proof}

This result shows us that the reflection of a single root about the complex unit circle will alter the $X_\gamma$ norm of a function in a predictable way, and will always decrease the $X_\gamma$ norm. We end this subsection with a Corollary involving functions with a single root in $\mathbb{D}$.

\begin{corollary}\label{thm:singlereflection}
Let $\gamma_n$ be a monotone increasing sequence with $\gamma_0=0$. If $F\in X_\gamma$ has a single root, $\alpha$, in $\mathbb{D}$, then
\begin{equation}
\|G\|_{X_\gamma} ^2= \|F\|_{X_\gamma} ^2 - (1-|\alpha|^2) \left\|\frac{G(e^{i\cdot})}{1-\overline{\alpha}e^{i\cdot}}\right\|_{Y_\gamma}^2 \ 
\end{equation}
\end{corollary}
\begin{proof}
We begin by noting that for all $z\in \mathbb{D}$,
\[F(z)=\frac{z-\alpha}{1-\overline{\alpha}z} G(z).\]
Therefore, we know that $\phi_\alpha(F)=G$, and $H_\alpha(z)= \frac{G(z)}{1-\overline{\alpha}z}$. Therefore, by applying Proposition~\ref{thm:reflectroot}, we have the result.
\end{proof}

With this result proved, we have completed this part of the section, and move on to functions with finitely many roots in $\mathbb{D}$.

\subsection{Part 2: Performing a Finite Number of Reflections}\label{Part2}
In the previous section we identified the relationship between $F$ and $G$, and studied how the $X_\gamma$ norm is affected by reflecting a single root across $\partial \mathbb{D}$ with Proposition \ref{thm:reflectroot}. Unfortunately, for functions with multiple zeros in $\mathbb{D}$, a single reflection will not produce the function $G$. Further, after we have performed a reflection, if we reflect a second root, we will be acting upon a new function. Therefore, to utilize the full potential of Proposition \ref{thm:reflectroot}, we require the following definition.
\begin{definition}\label{def:Fj}Let $F\in X_\gamma$ have $m$ roots in $\mathbb{D}$, enumerated $\alpha_1, \alpha_2 \dots \alpha_m$, in increasing order of magnitude.  Then expressing 
\[F_0(z):=F(z)=  \left(\prod_{j=1}^m \frac{z-\alpha_j}{1-\overline{\alpha_j}z}\right) \cdot G(z),\]
where $G$ has no zeros in $\mathbb{D}$, we define
\begin{eqnarray*}
F_{k}(z)&:=&  \left(\prod_{j=k+1}^m \frac{z-\alpha_j}{1-\overline{\alpha_j}z}\right) \cdot G(z) \quad \text{ where }1\leq k<m \\
F_{m}(z) &:=&G(z). \\
\end{eqnarray*}

\end{definition}

With this definition and using the same notation as Definition \ref{def:Halphaphi}
, we notice that for each $1\leq k \leq m,$
 \begin{equation}
\phi_{\alpha_{k}}(F_{k-1}) = F_{k}, \quad H_{\alpha_k}(z):=\frac{1}{1-\overline{\alpha_k}z} F_k(z).
\end{equation}

With all of this, we can prove a simple, yet useful lemma.
\begin{lemma}\label{thm:nreflections}
Suppose that $F$ has $m$ roots in $\mathbb{D}$, labeled in increasing order of magnitude $\alpha_1,\dots \alpha_m$. Then for any $1\leq n \leq m$, we have the identity
\begin{equation}
\|F_n\|_{X_\gamma}^2  = \|F\|_{X_\gamma}^2 - \sum_{j=1}^n \left( (1-|\alpha_j|^2) \left \| H_{\alpha_j}\right \|_{Y_\gamma}^2 \right).\label{eq:mainidentity}
\end{equation}
\end{lemma}

\begin{proof}
We begin by observing that for $0\le k \le n-1$, we have $\phi_{\alpha_{k+1}}(F_k) = F_{k+1}$. Then by applying Proposition~\ref{thm:reflectroot} to $F_k$, for $0\leq k \leq m-1,$  we have the identity 
\begin{equation}\label{eq:firstreflect}
\|F_{k+1}\|_{X_\gamma}^2 = \|F_{k}\|_{X_\gamma}^2 - (1-|\alpha_k|^2)\|H_{\alpha_k}\|_{Y_\gamma}^2.
\end{equation}
 
Now, by applying Equation~\eqref{eq:firstreflect} to each $F_k$, we have
\begin{align}
\|F_{1}\|_{X_\gamma}^2 &= \|F\|_{X_\gamma}^2 - (1-|\alpha_1|^2)\|H_{\alpha_1}\|_{Y_\gamma}^2\\
\|F_2\|_{X_\gamma}^2 &=  \left(\|F\|_{X_\gamma}^2 - (1-|\alpha_1|^2)\|H_{\alpha_1}\|_{Y_\gamma}^2\right) - (1-|\alpha_2|^2)\|H_{\alpha_2}\|_{Y_\gamma}^2 \\
&\vdots\nonumber  \\
\|F_{n-1}\|_{X_\gamma}^2  &= \|F\|_{X_\gamma}^2 - \sum_{j=1}^{n-1} \left( (1-|\alpha_j|^2) \left \| H_{\alpha_j}\right \|_{Y_\gamma}^2 \right)\\
\|F_n\|_{X_\gamma}^2  &= \|F\|_{X_\gamma}^2 - \sum_{j=1}^n \left( (1-|\alpha_j|^2) \left \| H_{\alpha_j}\right \|_{Y_\gamma}^2 \right).
\end{align}
\end{proof}
As a direct consequence of this Lemma, we now have an identity for any function $F\in X_\gamma$ with $m$ many roots in $\mathbb{D}$:
\begin{equation}
\|G\|_{X_\gamma}^2  = \|F\|_{X_\gamma}^2 - \sum_{j=1}^m \left( (1-|\alpha_j|^2) \left \| H_{\alpha_j}\right \|_{Y_\gamma}^2 \right). \label{eq:mainidentity}
\end{equation}

With this identity established, we now look to restrict our choices of $\gamma_n$ to create a bound on the terms $ \|H_{\alpha_j}\|_{Y_\gamma}^2$ involving the functions $F$ and $G$. 

We begin by noticing that the $Y_\gamma$ semi-norm will carry the same properties as the $X_\gamma$ norm if the sequence
\[\Gamma_n:= \gamma_{n+1}-\gamma_n \]
is monotone increasing. In other words, we can treat $Y_\gamma $ as $X_{\Gamma}$. This means that any inequalities that can be applied to the $X_\gamma$ norm can be applied to the $Y_\gamma$ norm. Namely, for any function $F\in Y_\gamma$, we can apply the results of Proposition~\ref{thm:reflectroot} to see
\begin{equation}\label{eq:Ygammaineq1}
\|F\|_{Y_\gamma}^2-\|\phi_{\alpha}(F)\|_{Y_\gamma}^2 \geq 0 .
\end{equation}

If the bounded sequence $\Gamma_n$ is monotone decreasing, then $\Gamma_0-\Gamma_n$ will be monotone increasing and nonnegative. This tells us that 
\[\|F\|_{X_{\Gamma_0-\Gamma_n}}^2-\|\phi_{\alpha}(F)\|_{X_{\Gamma_0-\Gamma_n}}^2 \geq 0.  \]
By the structure of the $X_\gamma$ norm space, and the fact that $\|F\|_{X_{\Gamma_0}}^2 = \|\phi_\alpha(F)\|_{X_{\Gamma_0}}^2$, we directly obtain the inequality
\begin{equation}\label{eq:Ygammaineq2}
\|\phi_{\alpha}(F)\|_{Y_\gamma}^2 -\|F\|_{Y_\gamma}^2 \geq 0.
\end{equation}

These inequalities help us to prove the following lemma.

\begin{lemma} \label{thm:Ygammabound}
Let $\gamma_n$ be a monotone increasing sequence satisfying $\gamma_0=0$, and let $F\in X_\gamma$ have $m$ roots in $\mathbb{D}$ labeled in increasing order of magnitude, $\alpha_j,$ where $m$ can be finite or infinity.  
\begin{enumerate}
\item{If $\gamma_n$ satisfies \eqref{eq:increase}, then for any $1<k< m$,
\[\left\|\frac{G(e^{i\cdot})}{1-\overline{\alpha_k}e^{i\cdot}}\right\|_{Y_\gamma}^2 \le \|H_{\alpha_k}\|_{Y_\gamma}^2 \le \left\|\frac{F(e^{i\cdot})}{e^{i\cdot}-\alpha_k}\right\|_{Y_\gamma}^2.   \]}

\item{If $\gamma_n$ satisfies \eqref{eq:decrease}, then  for any $ 1<k< m$,
\[ \left\|\frac{F(e^{i\cdot})}{e^{i\cdot}-\alpha_k}\right\|_{Y_\gamma}^2 \le \|H_{\alpha_k}\|_{Y_\gamma}^2 \le \left\|\frac{G(e^{i\cdot})}{1-\overline{\alpha_k}e^{i\cdot}}\right\|_{Y_\gamma}^2       .  \]}
\end{enumerate}
\end{lemma}

\begin{proof}
To begin, let
\[F(z) = \prod_{j=0}^ m \frac{\alpha_j-z}{1-\overline{\alpha_j}z} G(z).\]
We know that for any $1<k<m$, we can express
\[H_{\alpha_k} = \left(\prod_{j=k+1}^m  \frac{\alpha_j-z}{1-\overline{\alpha_j}z} \right)\cdot \left(\frac{1}{1-\overline{\alpha_k}z} G(z)\right).\]
Similarly, by rearranging, we know that
\[\frac{F(z)}{\alpha_k-z}= \left(\prod_{\substack{j=0\\j\neq k}}^ m \frac{\alpha_j-z}{1-\overline{\alpha_j}z}\right) \cdot \frac{1}{1-\overline{\alpha_k}z}G(z).\]
Thus, by reflecting the first $k-1$ roots of $F$ across the unit circle, we get $H_{\alpha_k}$. By reflecting the remaining roots, we get $\frac{1}{1-\overline{\alpha_k}\cdot}G(\cdot)$.

If $\Gamma_n$ is monotone increasing, then by \eqref{eq:Ygammaineq1}, we know that 
\[ \left\|\frac{F(z)}{\alpha_k-z}\right\|_{Y_\gamma}^2 \geq \left\|\phi_{\alpha_1}\left(\frac{F(z)}{\alpha_k-z}\right)\right\|_{Y_\gamma}^2 \geq \dots \geq \|H_{\alpha_k}\|_{Y_\gamma}^2 \geq \|\phi_{\alpha_{k+1}}(H_{\alpha_k})\|_{Y_\gamma}^2 \geq \dots.\]
Clearly, if $m$ is finite the result holds. If $m=\infty$, by the monotonicity of the sequence of $Y_\gamma$ norms, we know that this implies that
\[\|H_{\alpha_k}\|_{Y_\gamma}^2 \geq \left\|\frac{G(e^{i\cdot})}{1-\overline{\alpha_k}e^{i\cdot}}\right\|_{Y_\gamma}^2, \]
which proves the first part of the inequality.

Proving the second inequality follows verbatim with all inequalities flipped, due to  \eqref{eq:Ygammaineq2}. This gives the result.
\end{proof}
With this lemma, we can see that by restricting $\gamma_n$ to be a sequence that is either increasing at a non-increasing or non-decreasing rate, we can replace the intermediary $H_{\alpha_j}$ terms in Equation~\eqref{eq:mainidentity}. With all of this, we are now able to prove Theorem \ref{thm:MR1} and Theorem \ref{thm:MR2}. 

\subsubsection{Proofs of Theorem \ref{thm:MR1},  Corollary \ref{thm:MR1Cor1}, Corollary  \ref{thm:MR1Cor2}, and Theorem \ref{thm:MR2}, }
With the previous results, we now prove Theorem \ref{thm:MR1}, Theorem \ref{thm:MR2}, and Corollary 2. 

We begin with the proof of Theorem \ref{thm:MR1} . As a reminder, the hypotheses of this theorem are that the function $F\in X_\gamma$ has a finite number of roots in $\mathbb{D}$ and the sequence $\Gamma_n$ is monotone increasing. We have to show that Equations \eqref{eq:thm1eq1} and \eqref{eq:thm1eq2} hold true.

\begin{proof} 
 Since $F$ has a finite number of roots in $\mathbb{D}$, by Lemma~\ref{thm:nreflections} we have the identity
 \[\|G\|_{X_\gamma}^2  = \|F\|_{X_\gamma}^2 - \sum_{j=1}^n \left( (1-|\alpha_j|^2) \left \| H_{\alpha_j}\right \|_{Y_\gamma}^2 \right). \]
 
 If $\Gamma_n$ is a monotone increasing sequence, then by Lemma \ref{thm:Ygammabound} we know that for all $1\leq j \leq m$, 
 \[ \left \| H_{\alpha_j}\right \|_{Y_\gamma}^2 \geq \left \| \frac{G(e^{i \cdot})}{1-\overline{\alpha_j}e^{i \cdot}}\right \|_{Y_\gamma}^2. \]
 Therefore, by replacing each $H_{\alpha_j}$ with $\frac{G(e^{i \cdot})}{1-\overline{\alpha_j}e^{i \cdot}}$, we preserve the inequality and have the result.
 
\end{proof}

\medskip

From here, we prove Corollary \ref{thm:MR1Cor1}.

\begin{proof}
Let $\Gamma_n \equiv C$. Then we know that both \eqref{eq:increase} and \eqref{eq:decrease} hold true. Therefore, by Lemma \ref{thm:Ygammabound}, for all $1\le j \le m$
 \[  \left\| \frac{F(e^{i \cdot})}{e^{i \cdot}-\alpha_j}\right \|_{Y_\gamma}^2= \| H_{\alpha_j} \|_{Y_\gamma}^2 =\left \| \frac{G(e^{i \cdot})}{1-\overline{\alpha_j}e^{i \cdot}}\right \|_{Y_\gamma}^2. \]
 Therefore, by replacing each term in \eqref{eq:mainidentity}, we have the result.
\end{proof}
 With this proved, we now prove Corollary \ref{thm:MR1Cor2}. 
   \medskip
   
   \begin{proof}
   By definition, we know that given $F\in W^{1,2}$, we can represent
   \[\|F\|_{W^{1,2}}^2 = \sum_{j=0}^\infty (1+j^2) |a_j|^2= \sum_{j=0}^\infty j^2 |a_j|^2  + \sum_{j=0}^\infty |a_j|^2.\]
   Letting $\gamma_n = n^2$, this gives us the identity
   \begin{equation}
   \|\cdot\|_{W^{1,2}}^2 = \|\cdot\|_{X_\gamma}^2 + \|\cdot\|_{\mathcal{H}^2}^2. \label{eq:W12}
   \end{equation}
   
   Since $\gamma_{n+1}-\gamma_n = 2n+1= 2(n+1)-1$, we have the identity
   \begin{equation}
    \|\cdot \|_{Y_\gamma}^2 = 2\|\cdot\|_{\mathcal{D}}^2 - \|\cdot\|_{\mathcal{H}^2}^2
    \end{equation}
      Therefore, by invoking Theorem 1 on $\gamma_n$, we have
   
   \begin{equation}
   \|G\|_{X_\gamma}^2 \leq \|F\|_{X_\gamma}^2 - \sum_{j=1}^m (1-|\alpha_j|^2) \left[2\left \| \frac{G(e^{i\cdot})}{1-\overline{\alpha_j}e^{i\cdot}}\right\|^2_{\mathcal{D}}  - \left \| \frac{G(e^{i\cdot})}{1-\overline{\alpha_j}e^{i\cdot}}\right\|^2_{\mathcal{H}^2}\right].
  \end{equation}
   Finally, by the fact that $\|G\|_{\mathcal{H}^2}^2 = \|F\|_{\mathcal{H}^2}^2 $, and \eqref{eq:W12}, we have the result.
   \end{proof}
   
In studying the proof of this Corollary, we can see that similar techniques can also be used to create bounds for the spaces $W^{s,2}$ where $s\in \mathbb{N}$. By expanding the terms $(1+j^2)^s$, and applying the same techniques, we can obtain similar results.

\medskip

We now move on to Theorem \ref{thm:MR2}. The hypotheses of this theorem are similar to the previous theorem, with the exception that the sequence $\Gamma_n$ is monotone decreasing. With these conditions, we show that Equation \eqref{eq:thm2eq1} holds.

\begin{proof}
Since $F$ has a finite number of roots in $\mathbb{D}$, by Lemma~\ref{thm:nreflections} we have the identity
 \[\|G\|_{X_\gamma}^2  = \|F\|_{X_\gamma}^2 - \sum_{j=1}^n \left( (1-|\alpha_j|^2) \left \| H_{\alpha_j}\right \|_{Y_\gamma}^2 \right). \]
 
 If $\Gamma_n$ is a monotone decreasing sequence, then by Lemma \ref{thm:Ygammabound} we know that for all $1\leq j \leq m$, 
 \[ \left \| H_{\alpha_j}\right \|_{Y_\gamma}^2 \geq \left \| \frac{F(e^{i \cdot})}{e^{i \cdot}-\alpha_j}\right \|_{Y_\gamma}^2. \]
 Therefore, by replacing each $H_{\alpha_j}$ with $\frac{F(e^{i \cdot})}{e^{i \cdot}-\alpha_j}$, we preserve the inequality and have the result.
\end{proof}

With all results proved in the case when the function $F$ has finite number of roots in $\mathbb{D}$, we now consider the case when $F$ has an infinite number of roots in $\mathbb{D}$.

\subsection{Part 3: Functions with Infinitely Many Roots in $\mathbb{D}$}\label{Part3}
Similar to the previous section, this section will study the relationship between the $X_\gamma$ norm of $F$ and $G$, with the exception that $F$ will now be assumed to have infinitely many zeros in $\mathbb{D}$. In this case, we rely on some well known literature about the convergence of the partial decompositions.

We know that $F$ has Blaschke decomposition $F=B\cdot G$, where $B$ is an infinite Blaschke Product, and $G\in \mathcal{H}^2$. Therefore, if we enumerate the roots of $F$ as $\alpha_1,\alpha_2,\dots$ and ignore the rotation term $e^{i \theta_0},$ for $\theta_0\in[0,2\pi)$, we can express
\begin{equation}
F(z) = \prod_{j=1}^\infty \frac{\alpha_j-z}{1-\overline{\alpha_j}z} G(z).
\end{equation}
Similarly to what was done in the last section, we want to investigate the relationship between the $X_\gamma$ norms of $F$ and $G$. As a fact that is proven in \cite{Ricci:2004}
, if we define $F_j$ as the partial decomposition of $F$ defined in Definition~\ref{def:Fj}, we know that $F_j \to G$ uniformly on compact subsets of $\mathbb{D}$. Further, by Proposition \ref{thm:nreflections}, we know that for any finite $m$, we have the identity
\begin{equation*}
\|F_m\|_{X_\gamma}^2 = \|F\|_{X_\gamma}^2 - \sum_{j=1}^m (1-|\alpha_j|^2)\|H_{\alpha_j}\|_{Y_\gamma}^2.
\end{equation*}
Therefore, by showing that 
\begin{align}
 \lim_{m\to \infty} \|F_m\|_{X_\gamma}^2 = \|G\|_{X_\gamma}^2 \quad \text{and } \label{eq:Xconvergence}\\
\sup_j \|H_{\alpha_j}\|_{Y_\gamma}^2<\infty, \label{eq:boundedsup}
  \end{align}
we will have a meaningful analogue to Equation~\eqref{eq:mainidentity} for functions with infinitely many zeros in $\mathbb{D}$, and can prove Theorem \ref{thm:MR3}.

We begin by proving a sufficient condition on $\gamma_n$ to ensure Equation \eqref{eq:Xconvergence} holds. 

\begin{lemma}\label{thm:FmtoG}
Suppose that $\gamma_n \nearrow M < \infty.$ Then 
\[ \lim_{m\to \infty} \|F_m\|_{X_\gamma}^2 = \|G\|_{X_\gamma}^2. \]
\end{lemma}
\begin{proof}
Since 
\[ \|F_m\|_{\mathcal{H}^2} \to \|G\|_{\mathcal{H}^2},\]
uniformly on compact subsets of $\mathbb{D}$, we know that for any fixed $r<1$, 
\[\lim_{m\to \infty} \int_0^{2\pi} |F_m(re^{i\theta})-G(re^{i\theta})|^2d\theta = 0. \]
Therefore, 
\[\sup_{0<r<1} \lim_{m\to \infty} \int_0^{2\pi} |F_m(re^{i\theta})-G(re^{i\theta})|^2d\theta = 0, \]
so 
\[\lim_{m\to \infty} \|F_m - G\|_{\mathcal{H}^2}^2 = 0.\]
By our assumption, $\gamma_n$ is bounded above by $M$, so we know
\[ \left| \|F_m\|_{X_\gamma}^2 - \|G\|_{X_\gamma}^2 \right| \leq  \|F_m - G\|_{X_\gamma}^2 \leq  M\|F_m - G\|_{\mathcal{H}}^2. \]
Therefore, the result holds.
\end{proof}

With this condition, we now have to restrict our choice of $\gamma_n$ to bounded sequences. Clearly no bounded sequence can increase at an increasing rate, so for the remainder of this section we assume that \eqref{eq:decrease} holds.

From here, we need to find a condition on $\gamma_n$ so that \eqref{eq:boundedsup} is true on the space $Y_\gamma$. We know from Lemma~\ref{thm:Ygammabound} and from our restriction on $\gamma_n$ that for any finite $j$, 
\[\left\|\frac{F(e^{i\cdot})}{e^{i\cdot}-\alpha_j}\right\|_{Y_\gamma}^2 \le \|H_{\alpha_j}\|_{Y_\gamma}^2 \le \left\|\frac{G(e^{i\cdot})}{1-\overline{\alpha_j}e^{i\cdot}}\right\|_{Y_\gamma}^2. \]

Therefore, if we can find a condition on $\gamma_n$ such that 
\[ \sup_j \left\|\frac{G(e^{i\cdot})}{1-\overline{\alpha_j}e^{i\cdot}}\right\|_{Y_\gamma}^2 < \infty,\]
we will also satisfy \eqref{eq:boundedsup}. The following lemma provides such a condition.

\begin{lemma}\label{thm:boundedGsup} Let $\{\gamma_n\} \nearrow M$ be a monotone increasing sequence where $\Gamma_n$ is monotone decreasing. If
\begin{equation}\label{eq:suffcond}
\sum_{j=0}^\infty M-\gamma_j < \infty,
\end{equation}

 then, for every function $F\in \mathcal{H}^2(\mathbb{D})$ with infinitely many roots labeled in increasing order of magnitude $\alpha_1,\alpha_2,\dots $, with Blaschke decomposition $F=B\cdot G$, we have
 \begin{equation}
 \sup_j \left \| \frac{G(z)}{1-\overline{\alpha_j}z} \right \|_{Y_\gamma}^2 < \infty.
 \end{equation}
 \end{lemma}
 
 \begin{proof}
 We break this proof into two cases, $0\leq |\alpha|<\frac{1}{2}$ and $\frac{1}{2} \leq |\alpha|<1$. In both cases we will provide a bound that is independent of the choice of $\alpha$, giving us the result. 
 
 To begin, let $\alpha$ be a root of $F$ with $|\alpha|<1$. We know by the decomposition theorem that
 \[\left\|\frac{F(z)}{z-\alpha}\right\|_{\mathcal{H}^2} = \left\|\frac{G(z)}{1-\overline{\alpha}z}\right\|_{\mathcal{H}^2},\]
 since
 \[F(z) = \frac{z-\alpha}{1-\overline{\alpha}z}\cdot \tilde B(z)\cdot G(z).\]
 If $0 \leq |\alpha| <\frac{1}{2}$, we know by Lemma \ref{thm:Hbound} that
 \[ \left\|\frac{G(z)}{1-\overline{\alpha}z}\right\|_{Y_\gamma}^2 \leq \gamma_1 \left\|\frac{G(z)}{1-\overline{\alpha}z}\right\|_{\mathcal{H}^2}^2 = \gamma_1\left\|\frac{F(z)}{z-\alpha}\right\|_{\mathcal{H}^2}^2 < 4\gamma_1 \|F(z)\|_{\mathcal{H}^2}^2. \]
 Since this bound is independent of $\alpha$, we are finished with this case.
 
 If $\frac{1}{2} \leq |\alpha| < 1$, then we can rewrite 
 \[\frac{G(z)}{1-\overline{\alpha}z} = \frac{-1}{\overline{\alpha}} \left(\frac{G(z)}{z-\frac{1}{\overline{\alpha}}}\right). \]
 We know by our bound on $\alpha$ that 
 \[ \left\| \frac{-1}{\overline{\alpha}} \left(\frac{G(z)}{z-\frac{1}{\overline{\alpha}}}\right) \right\|_{Y_\gamma}^2 \leq 4 \left\| \frac{G(z)}{z-\frac{1}{\overline{\alpha}}}\right\|_{Y_\gamma}^2. \]
 Therefore, if we bound the term
 \[ \left\| \frac{G(z)}{z-\frac{1}{\overline{\alpha}}}\right\|_{Y_\gamma}^2, \]

we will have completed the proof.

 For simplicity, we will denote $\beta = \frac{1}{\overline{\alpha}}$ where $1< |\beta| \leq 2$. We know that $G\in \mathcal{H}^2$, so we may express $G(z) = \sum_{n=0}^\infty c_n z^n$, for all $z\in \mathbb{D}$. Since $|\beta|>1$, we know that for all $z\in \mathbb{D}$, we can rewrite
 \[ G_\beta(z):=\frac{G(z)}{z-\beta} = \sum_{n=0}^\infty d_n z^n. \]
 
By Cauchy's formula for derivatives, we know that for each $n$,
\[d_n =\frac{G_\beta^{(n)}(0)}{n!}.\]
By the generalized Leibniz rule, we know that for any $|z| < 1$,
\[G_\beta^{(n)}(z) = \frac{d^n}{dz^n} \left[G(z) \cdot \frac{1}{z-\beta}\right] = \sum_{k=0}^n {n \choose k} \frac{d^{n-k}}{dz^{n-k}} [G(z)] \cdot \frac{d^k}{dz^k}\left[\frac{1}{z-\beta}\right].  \]
Since $G^{(n-k)}(0) = c_{n-k}(n-k)!$, and
\[ \frac{d^k}{dz^k}\left[\frac{1}{z-\beta}\right] = (k!)(-1)^k (z-\beta)^{-(k+1)},\]
we know
\[G_\beta^{(n)}(0) = \sum_{k=0}^n \frac{n!}{(n-k)! k!} c_{n-k}(n-k)! \cdot (k!) \frac{1}{\beta^{k+1}} = \frac{n!}{\beta^{n+1}} \sum_{k=0}^n c_k \beta^k .\]

Therefore, for each $0\leq n < \infty$, 
\[d_n = \frac{1}{\beta^{n+1}} \sum_{k=0}^n c_k \beta^k.\]

 From here, we consider the $Y_\gamma$ norm of $\frac{G(z)}{z-\beta}$.
  \[\left\| \frac{G(z)}{z-\beta}\right\|_{Y_\gamma}^2 = \sum_{n=0}^\infty (\gamma_{n+1}-\gamma_n) |d_n|^2 =\sum_{n=0}^\infty \frac{(\gamma_{n+1}-\gamma_n)}{|\beta|^{2n+2}} \left | \sum_{k=0}^n c_k \beta^{k}\right|^2. \]

By the Cauchy-Schwarz inequality we know 
\[\left| \sum_{k=0}^n c_k \beta^{k}\right|^2 \le \sum_{k=0}^n |c_k|^2 \sum_{k=0}^n |\beta^{k}|^2.\]
Therefore
\[\left\| \frac{G(z)}{z-\beta}\right\|_{Y_\gamma}^2 \leq \sum_{n=0}^\infty (\gamma_{n+1}-\gamma_n) \left(\sum_{k=0}^n |c_k|^2\right) \frac{ \sum_{k=0}^n |\beta^{k}|^2}{|\beta|^{2n+2}}. \]
Clearly, $ \sum_{k=0}^n |c_k|^2 \leq \|G\|_{\mathcal{H}^2}^2$. By finite geometric series and the bound $1<|\beta| \leq 2$ , we know
\[ \frac{1}{3}\leq \frac{ \sum_{k=0}^n |\beta^{k}|^2}{|\beta|^{2n+2}} < (n+1).\]
 
 Lastly, since $\gamma_n$ satisfies 
\[\sum_{n=0}^\infty M-\gamma_n< \infty,\]
we know that 
\[\sum_{n=0}^\infty (\gamma_{n+1}-\gamma_n) (n+1) = \sum_{n=0}^\infty (M-\gamma_{n}-(M-\gamma_{n+1})) (n+1) = \sum_{n=0}^\infty M-\gamma_n < \infty,\]
 Therefore, we have
\[\left\| \frac{G(z)}{z-\beta}\right\|_{Y_\gamma}^2 \leq \|G\|_{\mathcal{H}^2}^2 \sum_{n=0}^\infty (\gamma_{n+1}-\gamma_n) (n+1)<\infty. \]

 With this, we have found a bound for $\left\| \frac{G(z)}{z-\beta}\right\|_{Y_\gamma}^2$ that is independent of $\beta$, so the result is proven. 

Therefore we have shown
 \[\sup_j \left \| \frac{G(z)}{1-\overline{\alpha_j}z} \right \|_{Y_\gamma}^2 < \infty. \]
 \end{proof}

With this lemma, we now have a sufficient condition to show \eqref{eq:boundedsup}. An interesting observation about this lemma is the fact that it imposes a similar condition to the Blaschke condition on the sequence $\gamma_n$. That is, both
\[ \sum_{j=0}^\infty M-\gamma_j < \infty \quad \text{ and } \sum_{j=1}^\infty 1-|\alpha_j|<\infty \]
must hold true for our results.

With all of this, we can now prove Theorem \ref{thm:MR3}.

\subsubsection{Proof of Theorem \ref{thm:MR3}.}

\begin{proof}
We know by hypothesis that $F\in X_\gamma$, and the roots $\alpha_j$, for $j\in J$ satisfy the Blaschke condition, \eqref{eq:BlaschkeCondition}. By Proposition \ref{thm:nreflections}, we know that for any finite $n$, we have the identity
 \[\|F_n\|_{X_\gamma}^2  = \|F\|_{X_\gamma}^2 - \sum_{j=1}^n \left( (1-|\alpha_j|^2) \left \| H_{\alpha_j}\right \|_{Y_\gamma}^2 \right).\]
 
By Lemma \ref{thm:Ygammabound} and by Lemma \ref{thm:boundedGsup}, we know that
\[ \sup_{j\in J} \left\| \frac{F(e^{i\cdot})}{e^{i\cdot}-\alpha_j} \right \|_{Y_\gamma}^2 \leq \sup_{j\in J} \| H_{\alpha_j} \|_{Y_\gamma}^2\leq  \sup_{j\in J} \left\| \frac{G(e^{i\cdot})}{1-\overline{\alpha_j}e^{i\cdot}} \right \|_{Y_\gamma}^2 < \infty.\]
Further, we know that since all $|\alpha_j|<1$,
\[ \sum_{j=1}^\infty 1-|\alpha_j|^2 < 2  \sum_{j=1}^\infty 1-|\alpha_j| < \infty.\]
With this, we have
\[\sum_{j=1}^\infty (1-|\alpha_j|^2)\left \|\frac{F(e^{i\cdot})}{e^{i\cdot}-\alpha_j}  \right \|_{Y_\gamma}^2 < \infty.\]
Since the right hand summation converges, we know
 \[\lim_{n\to \infty}\|F_n\|_{X_\gamma}^2  \le  \|F\|_{X_\gamma}^2 - \sum_{j=1}^\infty \left( (1-|\alpha_j|^2) \left \|  \frac{F(e^{i\cdot})}{e^{i\cdot}-\alpha_j}  \right \|_{Y_\gamma}^2 \right).\]
 Since $\gamma_n$ is bounded, by Lemma \ref{thm:FmtoG}, we may pass the limit through the left hand side and have the inequality
 \[ \|G\|_{X_\gamma}^2  \le  \|F\|_{X_\gamma}^2 - \sum_{j=1}^\infty \left( (1-|\alpha_j|^2) \left \|  \frac{F(e^{i\cdot})}{e^{i\cdot}-\alpha_j}  \right \|_{Y_\gamma}^2 \right),\]
which proves the result.

\end{proof}

\end{document}